\documentclass[12pt,reqno]{amsart}
\usepackage{graphicx}
\usepackage{amsfonts,amsmath,amssymb}

\newtheorem{theorem}{Theorem}[section]
\newtheorem{lemma}[theorem]{Lemma}
\newtheorem{proposition}[theorem]{Proposition}
\newtheorem{corollary}[theorem]{Corollary}
\newtheorem{remark}[theorem]{Remark}
\newtheorem{definition}[theorem]{Definition}
\numberwithin{equation}{section}

\newcommand{\be}{\begin{equation}}
\newcommand{\ee}{\end{equation}}
\newcommand{\ba}{\begin{array}}
\newcommand{\ea}{\end{array}}
\newcommand{\bea}{\begin{eqnarray}}
\newcommand{\eea}{\end{eqnarray}}
\newcommand{\bee}{\begin{eqnarray*}}
\newcommand{\eee}{\end{eqnarray*}}

\newcommand{\lab}{\label}

\newcommand{\ds}{\displaystyle}
\newcommand{\nn}{\nonumber}
\providecommand{\norm}[1]{\lVert#1\rVert}
\providecommand{\normm}[1]{\left\lVert#1\right\rVert}
\renewcommand{\c}{\cdot}
\newcommand{\les}{\lesssim}

\newcommand{\la}{\lambda}

\newcommand{\dd}{{\bf D}}

\newcommand{\R}{\mathbb{R}}

\newcommand{\MM}{\mathcal{M}}

\renewcommand{\gg}{{\bf g}}

\renewcommand{\th}{\theta}
\newcommand{\ep}{\varepsilon}

\renewcommand{\a}{\alpha}
\renewcommand{\b}{\beta}

\newcommand{\Si}{\Sigma}
\newcommand{\Sit}{\Sigma_t}
\newcommand{\pr}{\partial}
\newcommand{\nab}{\nabla}

\renewcommand{\o}{\omega}
\renewcommand{\S}{\mathbb{S}^2}
\newcommand{\po}{\partial_{\omega}}
\newcommand{\xo}{x\cdot\omega}

\renewcommand{\H}{\mathcal{H}}

\renewcommand{\a}{\alpha}
\renewcommand{\b}{\beta}

\newcommand{\ga}{\gamma}

\begin{document}

\title[Sharp Strichartz estimates for the wave equation on a rough background]{Sharp Strichartz estimates for the wave equation on a rough background}

\author{Jeremie Szeftel}
\address{DMA, Ecole Normale SupŽrieure, Paris 75005}
\email{Jeremie.Szeftel@ens.fr}
\vspace{-0.3in}

\maketitle

\begin{abstract}
In this paper, we obtain sharp Strichartz estimates for solutions of the wave equation $\square_\gg\phi=0$ where $\gg$ is a rough Lorentzian metric on a 4 dimensional space-time $\MM$. This is the last step of the proof of the bounded $L^2$ curvature conjecture proposed in \cite{Kl:2000}, and solved by S. Klainerman, I. Rodnianski and the author in \cite{boundedl2}, which also relies on the sequence of papers \cite{param1} \cite{param2} \cite{param3} \cite{param4}. Obtaining such estimates is at the core of the low regularity well-posedness theory for quasilinear wave equations. The difficulty is intimately connected to the regularity of the Eikonal equation $\gg^{\a\b}\pr_\a u\pr_\b u=0$ for a rough metric $\gg$. In order to be consistent with the final goal of proving the bounded $L^2$ curvature conjecture, we prove Strichartz estimates for all admissible Strichartz pairs under minimal regularity assumptions on the solutions of the Eikonal equation. 
\end{abstract}

\section{Introduction}

In this paper, we obtain sharp Strichartz estimates for solutions of the wave equation $\square_\gg\phi=0$ where $\gg$ is a rough Lorentzian metric on a 4 dimensional space-time $\MM$. This is the last step of the proof of the bounded $L^2$ curvature conjecture proposed in \cite{Kl:2000}, and solved by S. Klainerman, I. Rodnianski and the author in \cite{boundedl2}, which also relies on the sequence of papers \cite{param1} \cite{param2} \cite{param3} \cite{param4}. Obtaining such estimates is at the core of the low regularity well-posedness theory for quasilinear wave equations. The difficulty is intimately connected to the regularity of the Eikonal equation $\gg^{\a\b}\pr_\a u\pr_\b u=0$ for a rough metric $\gg$. In order to be consistent with the final goal of proving the bounded $L^2$ curvature conjecture, we prove Strichartz estimates for all admissible Strichartz pairs under minimal regularity assumptions on the solutions of the Eikonal equation. 

Since we are ultimately interested in local well-posedness, it is enough to prove local in time Strichartz estimates.  Also, it is natural to prove Strichartz estimates which are localized in frequency\footnote{The standard proof of Strichartz estimates in the flat case proceeds in two steps (see for example \cite{sog}). First, one localizes in frequency using Littlewood-Paley theory. Then, one proves the corresponding Strichartz estimates localized in frequency.}. Thus, we focus in this paper on the issue of proving local in time Strichartz estimates which are localized in frequency. In particular, this turns out to be sufficient for the proof of the bounded $L^2$ curvature conjecture.\\

We start by recalling the sharp Strichartz estimates for the standard wave equation on $(\R^{1+3}, \bf{m})$ where $\bf{m}$ is the Minkowski metric. We consider $\phi$ solution of
\be\lab{waveflat}
\left\{\ba{l}
\square\phi=0,\,(t,x)\in\R^+\times\R^3\\
\phi(0,.)=\phi_0,\, \pr_t\phi(0,.)=\phi_1,
\ea\right.
\ee
where 
$$\square=\square_{\bf{m}}=-\pr_t^2+\Delta_x.$$
Let $(p,q)$ such that $p, q\geq 2$, $q<+\infty$, and 
$$\frac{1}{p}+\frac{1}{q}\leq\frac{1}{2}.$$
Let $r$ defined by
$$r=\frac{3}{2}-\frac{1}{p}-\frac{3}{q}.$$
We call $(p,q,r)$ an admissible pair. Then, the solution $\phi$ of \eqref{waveflat} satisfies the following estimates, called Strichartz estimates \cite{Stri1} \cite{Stri2}
\be\lab{strichflat}
\norm{\phi}_{L^p(\R^+, L^q(\R^3))}\les \norm{\phi_0}_{H^r(\R^3)}+\norm{\phi_1}_{H^{r-1}(\R^3)}.
\ee

Strichartz estimates allow to obtain well-posedness results for nonlinear wave equations with less regularity for the Cauchy data $(\phi_0, \phi_1)$ than what is typically possible by relying only on energy methods (see for example  \cite{PoSi} in the context of semilinear wave equations). Therefore, as far as low regularity well-posedness theory for quasilinear wave equations is concerned, a considerable effort was put in trying to derive Strichartz estimates for the wave equation
\be\lab{wavecurve}
\square_\gg\phi=0
\ee
on a space-time $(\MM,\gg)$ where $\gg$ has limited regularity, see \cite{Sm}, \cite{Ba-Ch1}, \cite{Ba-Ch2}, \cite{Ta1}, \cite{Ta2}, \cite{Ta}, \cite{Kl-R1}, \cite{KR:Annals}, \cite{SmTa}. All these methods have in common a crucial and delicate analysis of the regularity of solutions $u$ to the Eikonal equation
$$\gg^{\a\b}\pr_\a u\pr_\b u=0.$$

To illustrate the role played by the Eikonal equation, let us first recall the plane wave representation of the standard wave equation. The solution $\phi$ of \eqref{waveflat} is given by:
\begin{equation}\label{flatparam}
\begin{array}{l}
\ds\int_{\S}\int_0^{+\infty}e^{i(-t+\xo)\la}\frac{1}{2}\left(\mathcal{F}\phi_0(\la\o)+i\frac{\mathcal{F}\phi_1(\la\o)}{\la}\right)\la^2d\la d\o\\
\ds +\int_{\S}\int_0^{+\infty}e^{i(t+\xo)\la}\frac{1}{2}\left(\mathcal{F}\phi_0(\la\o)-i\frac{\mathcal{F}\phi_1(\la\o)}{\la}\right)\la^2d\la d\o,
\end{array}
\end{equation}
where $\mathcal{F}$ denotes the Fourier transform on $\R^3$. The plane wave representation \eqref{flatparam} is the sum of two half waves, and Strichartz estimates are derived for each half-wave separately with an identical proof so we may focus on the first half-wave which we rewrite under the form
\be\lab{flathalfwave}
\int_{\S}\int_0^{+\infty}e^{i(-t+\xo)\la}f(\la\o)\la^2d\la d\o
\ee
where the function $f$ on $\R^3$ is explicitly given in term of the Fourier transform of the initial data. Note that $-t+x\c\o$ is a family of solutions to the Eikonal equation in the Minkowski space-time depending on the extra parameter $\o\in\S$. The natural generalization of \eqref{flathalfwave} to the curved case is the following representation formula - also called parametrix
\be\lab{curvedparam}
\int_{\S}\int_0^{+\infty}e^{i\la u(t,x,\o)}f(\la\o)\la^2d\la d\o
\ee
where $u$ is a family of solutions to the Eikonal equation in the curved space-time $(\MM, \gg)$ depending on the extra parameter $\o\in\S$. Thus, our parametrix is a Fourier integral operator with a phase $u$ satisfying the Eikonal equation\footnote{We refer to \cite{param2} \cite{param4} for a precise construction of a parametrix of the form \eqref{curvedparam} which generates any initial data of \eqref{wavecurve} and for its control in the context of the bounded $L^2$ curvature theorem of \cite{boundedl2}}.\\

Assume now that the space-time $\mathcal{M}$ is foliated by space-like hypersurfaces $\Sit$ defined as level hypersurfaces of a time function $t$. The estimate for the parametrix \eqref{curvedparam} corresponding to the Strichartz estimates of the flat case \eqref{strichflat} is
\be\lab{strichparam}
\normm{\int_{\S}\int_0^{+\infty}e^{i\la u(t,x,\o)}f(\la\o)\la^2d\la d\o}_{L^p(\R^+,L^q(\Sigma_t))}\les\norm{\la^rf}_{L^2(\R^3)}.
\ee
Since we are ultimately interested in local well-posedness, it is enough to restrict the time interval to $[0,1]$, which corresponds to local in time Strichartz estimates.  Also, it is natural to prove Strichartz estimates which are localized in frequency (see footnote 1). Thus we focus on proving Strichartz estimates on the time interval $[0,1]$ for a parametrix localized in a dyadic shell. Let $j\geq 0$, and let $\psi$ a smooth function on $\R^3$ supported in 
$$\frac 1 2 \leq |\xi|\leq 2.$$
Let $\varphi_j$ the scalar function on $\MM$ defined by the following oscillatory integral:
\be\lab{paraml}
\varphi_j(t,x)=\int_{\S} \int_0^\infty  e^{i \la u(t,x,\o)}\psi(2^{-j}\la)f(\la\o)\la^2d\la d\o.
\ee
We will prove the following version of \eqref{strichparam}, both localized in time and frequency
\be\lab{strichgeneralintro}
\norm{\varphi_j}_{L^p_{[0,1]}L^q(\Sigma_t)}\les 2^{jr}\norm{\psi(2^{-j}\la)f}_{L^2(\R^3)}.
\ee

The Strichartz estimates \eqref{strichgeneralintro} are a consequence of the oscillations of the phase $u$ of the Fourier integral operator $\varphi_j$. Thus, one should expect to have to perform integrations by parts to obtain \eqref{strichgeneralintro}. In turn, this requires $u$ to have enough regularity to be able to perform these integrations by parts. But of course, the rougher the space-time $(\MM,\gg)$ is, the less regularity one can extract from the solution $u$ to the Eikonal equation. Our goal is to prove \eqref{strichgeneralintro} in the context of the bounded $L^2$ curvature theorem obtained in \cite{boundedl2}. This forces us to make assumptions on $u$ which are compatible with the one derived in the companion papers \cite{param1} \cite{param3}. In particular, we may assume the following regularity for $u$
\be\lab{reguintro}
\pr_{t,x} u\in L^\infty,\, \pr_{t,x}\po u\in L^\infty.
\ee
Now, the standard procedure for proving \eqref{strichgeneralintro} - which we shall follow here - is to use the $TT^*$ argument to reduce \eqref{strichgeneralintro} to an $L^1$-$L^\infty$ estimate by interpolation, and finally to a $L^\infty$ estimate for an oscillatory integral with a phase involving $u$. One then typically uses the stationary phase to conclude the proof. This would require at the least\footnote{The regularity \eqref{staphase} is necessary to make sense of the change of variables involved in the stationary phase method (see Remark \ref{rem:compstatphase}).}
\be\lab{staphase}
\pr_{t,x}u\in L^\infty,\,\pr_{t,x}\po^2u\in L^\infty.
\ee
\eqref{staphase} involves one more derivative than our assumptions \eqref{reguintro} and we thus are forced to follow an alternative approach\footnote{This approach is inspired by the overlap estimates for wave packets derived in \cite{Sm} and \cite{SmTa} in the context of Strichartz estimates respectively for $C^{1,1}$ and $H^{2+\ep}$ metrics (see Remark \ref{reminiscentoverlap})} to the stationary phase in order to prove \eqref{strichgeneralintro} under the regularity assumption \eqref{reguintro} for $u$.

\vspace{0.5cm}

\noindent{\bf Acknowledgments.} The author wishes to express his deepest gratitude to Sergiu Klainerman and Igor Rodnianski for stimulating discussions and encouragements. He also would like to stress that the way this paper fits  into the whole proof of the bounded $L^2$ curvature conjecture has been done in collaboration with them. The author is supported by ANR jeunes chercheurs SWAP.

\section{Assumptions on the phase $u(t,x,\o)$ and main results}

\subsection{Time foliation on $\mathcal{M}$}

We foliate the space-time $\mathcal{M}$ by space-like hypersurfaces $\Sit$ defined as level hypersurfaces of a time function $t$. We consider local in time Strichartz estimates. Thus we may assume $0\leq t\leq 1$ so that
\be\lab{decopmpoM}
\MM=\bigcup_{0\leq t\leq 1}\Sigma_t.
\ee
We denote by 
$T$ the unit, future oriented, normal to $\Sigma_t$. We also define the lapse $n$ as 
\be\lab{lapsen}
n^{-1}=T(t).
\ee
Note that we have the following identity between the volume element of $\MM$ and the volume element corresponding to the induced metric on $\Sigma_t$
\be\lab{compvolume}
d\MM=n\,d\Sigma_t\,dt.
\ee
We will assume the following assumption on $n$
\be\lab{assonn}
\frac{1}{2}\leq n\leq 2
\ee
which together with \eqref{compvolume} yields
\be\lab{compvolume1}
d\MM\simeq d\Sigma_t\,dt.
\ee

\begin{remark}
The assumption \eqref{assonn} is very mild. Indeed, even for the very rough space-time $(\mathcal{M},\gg)$ constructed in \cite{boundedl2}, \eqref{assonn} is satisfied, and one has the additional regularity $\nabla n\in L^\infty$, where $\nabla$ denotes the induced covariant derivative on $\Sigma_t$. 
\end{remark}

\begin{remark}
In the flat case, we have $\mathcal{M}=(\R^{1+3},{\bf m})$,  where ${\bf m}$ is the Minkowski metric, and $\Sigma_t=\{t\}\times\R^3$. Also, $n=1$ so that $n$ satisfies \eqref{assonn} in this case.   
\end{remark}

\subsection{Geometry of the foliation generated by $u$ on $\mathcal{M}$}

Remember that $u$ is a solution to the eikonal equation $\gg^{\alpha\beta}\partial_\alpha u\partial_\beta u=0$ on $\mathcal{M}$ depending on a extra parameter $\o\in \S$.  The level hypersufaces $u(t,x,\o)=u$  of the optical function $u$ are denoted by  $\H_u$. Let $L'$ denote the space-time gradient of $u$, i.e.:
\be\lab{def:L'0}
L'=\gg^{\a\b}\pr_\b u \pr_\a.
\ee
Using the fact that $u$ satisfies the eikonal equation, we obtain:
\be\lab{def:L'1}
\dd_{L'}L'=0,
\ee
which implies that $L'$ is the geodesic null generator of $\H_u$.

We have: 
$$T(u)=\pm |\nab u|$$
where $|\nab u|^2=\sum_{i=1}^3|e_i(u)|^2$ relative to an orthonormal frame $e_i$ on $\Sigma_t$. Since the sign of $T(u)$ is irrelevant, we choose by convention:
\be\lab{it1'}
T(u)=-|\nab u|
\end{equation}
so that $u$ corresponds to $-t+x\c\o$ in the flat case.

Let
 \be\lab{it2}
L=bL'=T+N,
\end{equation}
where $L'$ is the space-time gradient of $u$ \eqref{def:L'0}, $b$  is  the  \textit{lapse of the null foliation} (or shortly null lapse)
\be\lab{it3}
b^{-1}=-<L', T>=-T(u),
\end{equation} 
and $N$ is a unit vectorfield given by
\be\lab{it3bis}
N=\frac{\nabla u}{|\nabla u|}.
\end{equation}

Note that we have the following identities
\begin{lemma}
\be\lab{identities}
L(u)=0,\, L(\po u)=0
\ee
and
\be\lab{identities1}
\gg(N, \po N)=0.
\ee
\end{lemma}

\begin{proof}
Using the definition \eqref{def:L'0} of $L'$ and the fact that $u$ satisfies the Eikonal equation, we have
$$L'(u)=\gg^{\a\b}\partial_\a u\partial_\b u=0.$$ 
In view of the definition \eqref{it2} of $L$, we deduce
\be\lab{gasp}
L(u)=0.
\ee
Also, differentiating the Eikonal equation with respect to $\o$ yields
$$\gg^{\a\b}\partial_\a u\partial_\b\po u=0$$
which yields
$$L'(\po u)=0$$
and thus
$$L(\po u)=0.$$
Together with \eqref{gasp}, this implies \eqref{identities}. 

Also, we have in view of the definition \eqref{it3bis} of $N$
$$\gg(N, N)=1.$$
Differentiating in $\o$, we obtain
$$\gg(N, \po N)=0,$$
which is \eqref{identities1}. This concludes the proof of the lemma.
\end{proof}

\subsection{Regularity assumptions for $u(t,x,\o)$}

We now state our assumptions for the phase $u(t,x,\o)$. These assumptions are compatible with the regularity obtained for the function $u(t,x,\o)$ constructed in \cite{param3}. Let $0<\ep<1$ a small enough universal constant\footnote{The fact that we may take $\ep$ small enough is consistent with the construction in \cite{param3} and results from a standard reduction to small data for proving well-posedness results for nonlinear wave equations (see \cite{boundedl2} for details on this procedure in the context of the bounded $L^2$ curvature theorem)}. $b$ and $N$ satisfy
\be\lab{regb}
\norm{b-1}_{L^\infty}+\norm{\po b}_{L^\infty}\les\ep.
\ee

\be\lab{regpoN}
\norm{\gg(\po N,\po N)-I_2}_{L^\infty}\les \ep.
\ee

\be\lab{ad1}
|N(., \o)-N(., \o')|=|\o-\o'|(1+O(\ep)).
\ee

\begin{remark}
In the flat case, we have $\mathcal{M}=(\R^{1+3},{\bf m})$,  where ${\bf m}$ is the Minkowski metric, $u(t,x,\o)=-t+\xo$, $b= 1$, $N=\o$ and $L=\partial_t+\o\cdot\partial_x$. Thus, the assumptions \eqref{regb} \eqref{regpoN} \eqref{ad1}  are clearly satisfied with $\ep=0$.
\end{remark}

\begin{remark}
In terms of the regularity of $u(t,x,\o)$, the assumptions \eqref{regb} \eqref{regpoN} correspond to 
$$\nabla u\in L^\infty\textrm{ and }\nabla\po u\in L^\infty$$
which is very weak. In particular, the classical proof for obtaining Strichartz estimates for the wave equation relies on the stationary phase for an oscillatory integral involving $u$ as a phase, and typically requires at the least one more derivative for $u$ (see Remark \ref{rem:compstatphase}).
\end{remark}

\subsection{A global coordinate system on $\Sigma_t$}\lab{sec:assumptioncoord}

For all $0\leq t\leq 1$, and for all $\o\in\S$, $(u(t,x,\o), \po u(t,x,\o))$ is a global coordinate system on $\Sigma_t$. Furthermore, the volume element is under control in the sense that in this coordinate system, we have
\be\lab{assglobalcoordvol}
\frac{1}{2}\leq \sqrt{\det g}\leq 2
\ee
where $g$ is the induced metric on $\Sigma_t$, and where $\det g$ denotes the determinant of the matrix of the coefficients of $g$. 


\begin{remark}
In the flat case, we have $\Sigma_t=\{t\}\times\R^3$ and $u(t,x,\o)=-t+\xo$ so that $(u(t,x,\o), \po u(t,x,\o))$ is clearly a global coordinate system on $\Sigma_t$ and $\det g=1$ in this case. These assumptions are also satisfied by the function $u(t,x, \o)$ constructed in \cite{param3}. 
\end{remark}

\subsection{Main results}

We next state our main result concerning general Strichartz inequalities in mixed space-time norms
of the form $L^p_{[0,1]}L^q(\Sigma_t)$  defined as follows,
$$\|F\|_{L^p_{[0,1]}L^q(\Sigma_t)}=\left(\int_0^1 \|F(t,\cdot)\|_{L^p(\Si_t)}^pdt\right)^{\frac{1}{p}}.$$
 
\begin{theorem}\lab{mainth}
Let $(p,q)$ such that $p, q\geq 2$, $q<+\infty$, and 
$$\frac{1}{p}+\frac{1}{q}\leq\frac{1}{2}.$$
Let $r$ defined by
$$r=\frac{3}{2}-\frac{1}{p}-\frac{3}{q}.$$
Then, the parametrix localized at frequency $j$ defined in \eqref{paraml} satisfies 
the following Strichartz inequality
\be\lab{strichgeneral}
\norm{\varphi_j}_{L^p_{[0,1]}L^q(\Sigma_t)}\les 2^{jr}\norm{\psi(2^{-j}\la)f}_{L^2(\R^3)}.
\ee
\end{theorem}

We also obtain the following corollary which is needed in the proof of the bounded $L^2$ curvature conjecture \cite{boundedl2}.
\begin{corollary}\lab{cor:strichl4}
The parametrix localized at frequency $j$ defined in \eqref{paraml} satisfies 
the following $L^4(\MM)$ Strichartz inequalities
\be\lab{l4strichbis}
\norm{\varphi_j}_{L^4(\MM)}\les 2^{\frac{j}{2}}\norm{\psi(2^{-j}\la)f}_{L^2(\R^3)},
\ee
and 
\be\lab{l4strichbis:1}
\norm{\nabla\varphi_j}_{L^4(\MM)}\les 2^{\frac{3j}{2}}\norm{\psi(2^{-j}\la)f}_{L^2(\R^3)}.
\ee
Furthermore, assume that $u$ satisfies the following additional assumption
\be\lab{aditionalassumption}
\sup_{\omega\in\S, u_0\in\R}\norm{\nabla^2u}_{L^4(\,^{(\omega)}\mathcal{H}_{u_0})}\les 1,
\ee
where for $\o\in\S$ and $u_0\in\R$, $\,^{(\omega)}\mathcal{H}_{u_0}$ denotes the level hypersurface of $u(.,\omega)$
$$\,^{(\omega)}\mathcal{H}_{u_0}=\{(t,x)\,/\,u(t,x,\omega)=u_0\}.$$
Then \eqref{paraml} satisfies the following $L^4(\MM)$ Strichartz inequality
\be\lab{l4strichbis:2}
\norm{\nabla^2\varphi_j}_{L^4(\MM)}\les 2^{\frac{5j}{2}}\norm{\psi(2^{-j}\la)f}_{L^2(\R^3)}.
\ee
\end{corollary}

\begin{remark}
The additional regularity assumption \eqref{aditionalassumption} is compatible with the regularity obtained for the function $u(t,x,\o)$ constructed in \cite{param3}. Note that it also holds in the flat case since we have $u=-t+x\cdot\omega$ and hence $\nabla^2u=0$.
\end{remark}

\vspace{0.3cm}

The rest of the paper is organized as follows. In section \ref{sec:proofmainthbis}, we use the standard $TT^*$ argument to reduce the proof of Theorem \ref{mainth} and Corollary \ref{cor:strichl4} to an upper bound on the kernel   $K$ of a certain operator. This kernel is an oscillatory integral with a phase $\phi$. In section \ref{sec:estkernel}, we prove the upper bound on the kernel $K$ provided we have a suitable lower bound on $\phi$. Finally, in section \ref{sec:lowerb}, we prove the lower bound for $\phi$ used in section \ref{sec:estkernel}.

\section{Proof of Theorem \ref{mainth} and Corollary \ref{cor:strichl4}}\lab{sec:proofmainthbis} 
 
\subsection{Proof of Theorem \ref{mainth}} 

Let $a(t,x,\omega)$ a scalar function on $\MM\times\S$. Let $T_j$ be the operator, applied to functions $f\in L^2(\R^3)$,
\bea
\label{eq:op-Tj}
T_jf(t,x)= \int_{\S} \int_0^\infty  e^{i \la u(t,x,\o)}a(t,x,\omega)\psi(2^{-j}\la)f(\la\o)\la^2d\la d\o.
\eea

$T_j$ satisfies the following estimate.
\begin{proposition}\lab{prop:TT*argument}
Let $(p,q)$ such that $p, q\geq 2$, $q<+\infty$, and 
$$\frac{1}{p}+\frac{1}{q}\leq\frac{1}{2}.$$
Let $r$ defined by
$$r=\frac{3}{2}-\frac{1}{p}-\frac{3}{q}.$$
Assume that the scalar function $a$ satisfies
\be\lab{assbounda}
\norm{a}_{L^\infty}\les 1.
\ee
Then, the operator $T_j$ defined in \eqref{eq:op-Tj} satisfies 
the following Strichartz inequality
\be\lab{strichgeneral:bis}
\norm{T_jf}_{L^p_{[0,1]}L^q(\Sigma_t)}\les 2^{jr}\norm{\psi(2^{-j}\la)f}_{L^2(\R^3)}.
\ee
\end{proposition}

The proof of Proposition \ref{prop:TT*argument} is postponed to section \ref{sec:tts}. Let us now conclude the proof of Theorem \ref{mainth}. Note that $T_j$ satisfies
$$T_jf=\varphi_j\textrm{ if }a(t,x,\omega)=1\textrm{ for all }(t,x,\o)\in \MM\times\S$$
where $\varphi_j$ is the parametrix localized at frequency $2^j$ defined in \eqref{paraml}. Thus, the estimate \eqref{strichgeneral} follows immediately from \eqref{strichgeneral:bis}. This concludes the proof of Theorem \ref{mainth}. 

\subsection{Proof of Corollary \ref{cor:strichl4}}  

Note first that \eqref{l4strichbis} follows immediately from Theorem \ref{mainth} by choosing $p=q=4$ in \eqref{strichgeneral}, and noticing in view of \eqref{decopmpoM} and \eqref{compvolume1} that
\be\lab{44444}
L^4_{[0,1]}L^4(\Sigma_t)=L^4(\MM).
\ee

Next, we turn to the proof of the estimates \eqref{l4strichbis:1} and \eqref{l4strichbis:2} starting with the first one. 
In view of the definition \eqref{paraml} of $\varphi_j$, we have
\be\lab{444}
\nabla\varphi_j(t,x)=i2^j\int_{\S} \int_0^\infty  e^{i \la u(t,x,\o)}\nabla u(t,x,\o)(2^{-j}\la)\psi(2^{-j}\la)f(\la\o)\la^2d\la d\o.
\ee
Note that 
$$\nabla\varphi_j=i2^jT_jf,$$
with $\psi(\la)$ replaced by $\la \psi(\la)$, and with the choice
$$a(t,x,\o)=\nabla u(t,x,\o).$$
Since we have $\nabla u=b^{-1}N$ in view of \eqref{it1'}, \eqref{it3} and \eqref{it3bis}, we deduce from the assumption \eqref{regb} that 
$$\norm{a}_{L^\infty}\les\norm{\nabla u}_{L^\infty}\les \norm{b^{-1}}_{L^\infty}\les 1$$
so that $a$ satisfies the assumption \eqref{assbounda}. Thus, \eqref{strichgeneral:bis} with the choice $p=q=4$  yields in view of \eqref{44444}
$$\normm{\int_{\S} \int_0^\infty  e^{i \la u(t,x,\o)}\nabla u(t,x,\o)(2^{-j}\la)\psi(2^{-j}\la)f(\la\o)\la^2d\la d\o}_{L^4(\MM)}\les 2^{\frac{j}{2}}\norm{\psi(2^{-j}\la)f}_{L^2(\R^3)}.$$
Together with \eqref{444}, we obtain 
$$\norm{\nabla\varphi_j}_{L^4(\MM)}\les 2^{\frac{3j}{2}}\norm{\psi(2^{-j}\la)f}_{L^2(\R^3)}$$
which is the desired estimate \eqref{l4strichbis:1}.\\

Finally, we turn to the proof of the estimate \eqref{l4strichbis:2}. Differentiating \eqref{444}, we obtain
\bea
\nn\nabla_l\nabla_m\varphi_j(t,x)&=& -2^{2j}\int_{\S} \int_0^\infty  e^{i \la u(t,x,\o)}\nabla_lu(t,x,\o)\nabla_mu(t,x,\o)(2^{-j}\la)^2\psi(2^{-j}\la)f(\la\o)\la^2d\la d\o\\
\label{444:1}&& +i2^j\int_{\S} \int_0^\infty  e^{i \la u(t,x,\o)}\nabla_l\nabla_m u(t,x,\o)(2^{-j}\la)\psi(2^{-j}\la)f(\la\o)\la^2d\la d\o.
\eea
Next, we estimate the two terms in the right-hand side of \eqref{444:1} starting with the first one. Note that 
$$\int_{\S} \int_0^\infty  e^{i \la u(t,x,\o)}\nabla_lu(t,x,\o)\nabla_mu(t,x,\o)(2^{-j}\la)^2\psi(2^{-j}\la)f(\la\o)\la^2d\la d\o=T_jf,$$
with $\psi(\la)$ replaced by $\la^2\psi(\la)$, and with the choice
$$a(t,x,\o)=\nabla_lu(t,x,\o)\nabla_mu(t,x,\o).$$
Since we have $\nabla u=b^{-1}N$, we deduce from the assumption \eqref{regb} that 
$$\norm{a}_{L^\infty}\les\norm{\nabla u}^2_{L^\infty}\les \norm{b^{-2}}_{L^\infty}\les 1$$
so that $a$ satisfies the assumption \eqref{assbounda}. Thus, \eqref{strichgeneral:bis} with the choice $p=q=4$ yields in view of \eqref{44444}
\bea\lab{444:2}
\nn&&\normm{\int_{\S} \int_0^\infty  e^{i \la u(t,x,\o)}\nabla_lu(t,x,\o)\nabla_mu(t,x,\o)(2^{-j}\la)^2\psi(2^{-j}\la)f(\la\o)\la^2d\la d\o}_{L^4(\MM)}\\
&\les& 2^{\frac{j}{2}}\norm{\psi(2^{-j}\la)f}_{L^2(\R^3)}.
\eea

Next, we estimate the second term in the right-hand side of \eqref{444:1}. We have
\bea\lab{444:3}
&&\normm{\int_{\S} \int_0^\infty  e^{i \la u(t,x,\o)}\nabla_l\nabla_m u(t,x,\o)(2^{-j}\la)\psi(2^{-j}\la)f(\la\o)\la^2d\la d\o}_{L^4(\MM)}\\
\nn&\les& \int_{\S} \normm{\left(\int_0^\infty  e^{i \la u(t,x,\o)}(2^{-j}\la)\psi(2^{-j}\la)f(\la\o)\la^2d\la\right)\nabla_l\nabla_m u(t,x,\o)}_{L^4(\MM)} d\o\\
\nn&\les& \int_{\S} \normm{\int_0^\infty  e^{i \la u(t,x,\o)}(2^{-j}\la)\psi(2^{-j}\la)f(\la\o)\la^2d\la}_{L^4_u}\norm{\nabla^2u(.,\o)}_{L^4(\mathcal{H}_u)} d\o\\
\nn&\les& \int_{\S} \normm{\int_0^\infty  e^{i \la u(t,x,\o)}(2^{-j}\la)\psi(2^{-j}\la)f(\la\o)\la^2d\la}_{L^4_u} d\o,
\eea
where we used in the last inequality the assumption \eqref{aditionalassumption} on $\nabla^2u$. Now, we have
\bee
&&\normm{\int_0^\infty  e^{i \la u(t,x,\o)}(2^{-j}\la)\psi(2^{-j}\la)f(\la\o)\la^2d\la}_{L^4_u}\\
&\les &\normm{\int_0^\infty  e^{i \la u(t,x,\o)}(2^{-j}\la)\psi(2^{-j}\la)f(\la\o)\la^2d\la}^{\frac{1}{2}}_{L^\infty_u}\\
&&\times\normm{\int_0^\infty  e^{i \la u(t,x,\o)}(2^{-j}\la)\psi(2^{-j}\la)f(\la\o)\la^2d\la}^{\frac{1}{2}}_{L^2_u}\\
&\les&  \left(2^{\frac{j}{2}}\norm{\psi(2^{-j}\la)\la^2f}_{L^2_{\la}}\right)^{\frac{1}{2}}\norm{\psi(2^{-j}\la)\la^2f}_{L^2_{\la}}^{\frac{1}{2}}\\
&\les& 2^{\frac{j}{4}}\norm{\psi(2^{-j}\la)\la^2f}_{L^2_{\la}}
\eee
where we used Cauchy-Schwartz in $\la$ to evaluate the $L^\infty_u$ norm and Plancherel to evaluate the $L^2_u$ norm. In view of \eqref{444:3}, this yields
\bee
&&\normm{\int_{\S} \int_0^\infty  e^{i \la u(t,x,\o)}\nabla_l\nabla_m u(t,x,\o)(2^{-j}\la)\psi(2^{-j}\la)f(\la\o)\la^2d\la d\o}_{L^4(\MM)}\\
&\les& 2^{\frac{j}{4}}\int_{\S}\norm{\psi(2^{-j}\la)\la^2f}_{L^2_{\la}}d\o\\
&\les& 2^{\frac{5j}{4}}\norm{\psi(2^{-j}\la)f}_{L^2(\R^3)}
\eee
where we used Cauchy-Schwarz in $\o$ in the last inequality. Together with \eqref{444:1} and \eqref{444:2}, we finally obtain
$$\norm{\nabla^2\varphi_j}_{L^4(\MM)}\les 2^{\frac{5j}{2}}\norm{\psi(2^{-j}\la)f}_{L^2(\R^3)}$$
which is the desired estimate \eqref{l4strichbis:2}. This concludes the proof of Corollary \ref{cor:strichl4}.

\subsection{Proof of Proposition \ref{prop:TT*argument} (the $TT^*$ argument)}\lab{sec:tts}

We start with the following remark.
\begin{remark}\lab{rem:globalcoordsimple}
Fixing  a global system of coordinates  $x=(x^1, x^2, x^3)$ in $\Si_t$, such as the one described in section  \ref{sec:assumptioncoord}, we note in view of \eqref{assglobalcoordvol} that \eqref{strichgeneral:bis} is equivalent with the same inequality where the norm $L^q(\Sigma_t)$ on the  left-hand side is replaced by the corresponding  euclidean norm in the  given coordinates. More precisely we can assume from now on that
$$\norm{F}_{L^p_{[0,1]}L^q(\Sigma_t)}=\left(\int_0^1\left(\int_{\R^3}|F(t,x)|^q dx\right)^{\frac{p}{q}}dt\right)^{\frac{1}{q}}$$
which we will denote by a slight abuse of notation by
$$\norm{F}_{L^p_{[0,1]}L^q(\R^3)}.$$
Note also that in the $(t,x)$  coordinates $\MM=[0,1]\times \R^3$. 
\end{remark}

To prove Proposition \ref{prop:TT*argument}, we rely on  the standard  $TT^*$ argument for
 the Fourier integral operator \eqref{eq:op-Tj}. Note that  the operator  $T_j^*  $  takes real valued  functions $h$ on $\MM$ to complex valued functions on $\R^3$
$$T_j^*h(\la\o)=\psi(2^{-j} \la) \int_\MM a(s,y,\omega) e^{-i\la u(s,y, \o)} h(s,y) ds dy.$$
Therefore, the operator $U_j:=T_j T^*_j$ is given by the formula,
$$U_j h(t,x)=\int_{\S} \int_0^\infty \int_\MM e^{i \la u(t,x,\o)-i\la u(s,y,\o)}a(t,x,\omega)a(s,y,\omega)\psi(2^{-j}\la)^2h(s,y)\la^2d\la d\o   ds dy.$$
Note, in view of Remark \ref{rem:globalcoordsimple}, that \eqref{strichgeneral:bis} is equivalent to the following estimate
\be\lab{ttsineq}
\norm{ U_j h}_{L^p_{[0,1]}L^q(\R^3)}\les 2^{2jr}\norm{h}_{L^{p'}_{[0,1]}L^{q'}(\R^3)},
\ee
where $p'$ (resp. $q'$) is the conjugate exponent to $p$ (resp. $q$).
Observe that,
\bee
U_jh\left(\frac{t}{2^j},\frac{x}{2^j}\right) &=& 2^{-j}\int_{\S} \int_0^\infty \int_{2^j\MM} e^{i \la 2^j u\left(\frac{t}{2^j},\frac{x}{2^j},\o\right)-i\la 2^j u\left(\frac{s}{2^j},\frac{y}{2^j},\o\right)}a\left(\frac{t}{2^j},\frac{x}{2^j},\omega\right)a\left(\frac{s}{2^j},\frac{y}{2^j},\omega\right)\\
&&\times\psi(\la)^2h\left(\frac{s}{2^j},\frac{y}{2^j}\right)\la^2d\la d\o  ds dy
\eee
with $2^j\MM=[0,  2^j]\times \R^3$ relative to the rescaled  variables $(s,y)$.  Thus,
 setting,
 \bee
 Ah(t,x) :=& \ds\int_{\S} \int_0^\infty \int_{2^j\MM} e^{i \la 2^j u\left(\frac{t}{2^j},\frac{x}{2^j},\o\right)-i\la 2^j u\left(\frac{s}{2^j},\frac{y}{2^j},\o\right)}a\left(\frac{t}{2^j},\frac{x}{2^j},\omega\right)a\left(\frac{s}{2^j},\frac{y}{2^j},\omega\right)\\
& \times\psi(\la)^2h(s,y)\la^2d\la d\o  ds dy
\eee
we have 
$$U_jh\left(\frac{t}{2^j},\frac{x}{2^j}\right)=2^{-j}A h_j(t,x),\,\,\, h_j(s,y)=h\left(\frac{s}{2^j},\frac{y}{2^j}\right).$$ 
We easily infer that \eqref{ttsineq} is equivalent to  the estimate,
\be\lab{ttsineq1}
\norm{Ah}_{L^p_{[0,2^j]}L^q(\R^3)}\les \norm{h}_{L^{p'}_{[0,2^j]}L^{q'}(\R^3)}.
\ee

We introduce the kernel $K$ of $A$
\bea\lab{defkernel}
K(t,x,s,y)&=&\int_{\S} \int_0^\infty e^{i \la 2^j u\left(\frac{t}{2^j},\frac{x}{2^j},\o\right)-i\la 2^j u\left(\frac{s}{2^j},\frac{y}{2^j},\o\right)}a\left(\frac{t}{2^j},\frac{x}{2^j},\omega\right)a\left(\frac{s}{2^j},\frac{y}{2^j},\omega\right)\\
\nn&&\times\psi(\la)^2\la^2d\la d\o.
\eea

\begin{remark}
In the flat case, we have $u(t,x,\o)=-t+x\cdot\o$ so that 
$$2^ju\left(\frac{t}{2^j},\frac{x}{2^j},\o\right)=u(t,x,\o).$$
In particular, in the case $a=1$, $K$ is independent of $j$ 
$$K(t,x,s,y)=\int_{\S} \int_0^\infty e^{i \la u(t,x,\o)-i\la u(s,y,\o)}\psi(\la)^2\la^2d\la d\o.$$
\end{remark}

We have the following proposition.
\begin{proposition}\lab{prop:estkernel}
The kernel $K$ of the operator $A$ satisfies the  dispersive estimates,
\be\lab{estkernel} 
|K(t,x,s,y)|\les \frac{1}{|t-s|},\,\,\,\forall (t,x)\in 2^j\MM,\,\,\, \forall (s,y)\in 2^j\MM. 
\ee
\end{proposition}
The proof of Proposition \ref{prop:estkernel} is postponed to section \ref{sec:estkernel}. We now conclude the proof of Proposition \ref{prop:TT*argument}. \eqref{ttsineq1} follows from \eqref{estkernel} using interpolation and the Hardy-Littlewood inequality according to the standard procedure, see for example \cite{sog} and \cite{stein}. Finally, in view of the discussion above, \eqref{ttsineq1} yields \eqref{ttsineq} which in turn implies \eqref{strichgeneral:bis}. This concludes the proof of Proposition \ref{prop:TT*argument}.

\section{Proof of Proposition \ref{prop:estkernel} (bound on the kernel $K$)}\lab{sec:estkernel}

Let $\phi$ the scalar function on $\MM\times\MM\times\S$ defined as
\be\lab{defphi}
\phi(t,x,s,y,\o)=u(t,x,\o)-u(s,y,\o).
\ee
In view of \eqref{defkernel}, we may rewrite $K$ as 
$$K(t,x,s,y)=\int_{\S} \int_0^\infty e^{i \la 2^j \phi\left(\frac{t}{2^j},\frac{x}{2^j},\frac{s}{2^j},\frac{y}{2^j},\o\right)}a\left(\frac{t}{2^j},\frac{x}{2^j},\omega\right)a\left(\frac{s}{2^j},\frac{y}{2^j},\omega\right)\psi(\la)^2\la^2d\la d\o.$$
After integrating by parts twice in $\la$, and using the assumption \eqref{assbounda} on $a$ and the size of the support of $\psi$, this yields
\be\lab{ibpkernel}
|K(t,x,s,y)|\les \int_{\S}\frac{1}{1+2^{2j}\phi\left(\frac{t}{2^j},\frac{x}{2^j},\frac{s}{2^j},\frac{y}{2^j},\o\right)^2}d\o.
\ee
The next section is dedicated to the obtention of a lower bound on $|\phi|$ which will allow us to deduce \eqref{estkernel} from \eqref{ibpkernel}.

\begin{remark}\lab{rem:compstatphase}
It is at this stage that we depart from the standard strategy for proving Strichartz estimates. Indeed, the usual method  consists in using the stationary phase method to derive \eqref{estkernel}. To this end, one needs an identity of the type
\be\lab{usualstatmeth}
\phi=(s-t)A(\o-\o_0)\c (\o-\o_0)+o\left((s-t)(\o-\o_0)^2\right)
\ee 
for $\o$ in the neighborhood of some $\o_0\in\S$ and for some $3\times 3$ invertible matrix $A$. \eqref{usualstatmeth} then allows to perform a change of variables in $\o$ which ultimately leads to \eqref{estkernel}. In particular, the standard method requires at the least 
$$\partial_{t,x}\po^2 u\in L^\infty$$
just to derive \eqref{usualstatmeth}. Our assumptions correspond only to 
$$\partial_{t,x}\po u\in L^\infty.$$
Thus, in order to obtain \eqref{estkernel}, we instead integrate by parts in $\la$ to obtain \eqref{ibpkernel}, and then look for a suitable lower bound on $|\phi|$. In particular, we obtain lower bounds of the following type (see details in Lemma \ref{lemma:key})
\be\lab{usualstatmeth1}
|\phi|\gtrsim |s-t||\o-\o_0|^2
\ee 
for $\o$ in the neighborhood of some $\o_0\in\S$. The fundamental observation is that the inequality \eqref{usualstatmeth1} requires less regularity than the equality \eqref{usualstatmeth}.
\end{remark}

\subsection{The key lemma}

Let $(t,x)$ and $(s,y)$ in $\MM$, and let $\o\in\S$. In this section, we obtain a lower bound on $\phi(t,x,s,y,\o)$. We may assume
$$0\leq t<s\leq 1.$$
\begin{definition}
For any $\o\in\S$ and $\sigma\in\R$, let $\gamma_{\o}(\sigma)$ denote the null geodesic parametrized by the time function and with initial data
$$\gamma_\o(0)=(t,x),\, \gamma_\o'(0)=b^{-1}(t,x,\o)L(t,x,\o).$$
\end{definition}

Recall from \eqref{def:L'1} and \eqref{it2} that $b^{-1}L$ is geodesic. Thus, for any $\o\in\S$ and any $\sigma\in\R$, we have
\be\lab{idgeod}
u(\gamma_\o(\sigma),\o)=u(t,x,\o),\, \gamma_\o'(\sigma)=b^{-1}(\gamma_\o(\sigma),\o)L(\gamma_\o(\sigma),\o).
\ee

\begin{definition}
Let us define the subset $S$ of $\Sigma_s$ as
\be\lab{def:S}
S=\bigcup_{\o\in\S}\{\ga_\o(s-t)\}.
\ee
\end{definition}

We also define for all $(s,z)\in\Sigma_s$
\be\lab{def:m}
m(s,z)=\max_{\o\in\S}(u(s,z,\o)-u(t,x,\o)).
\ee
We have the following lemma characterizing the zeros of $m$.

\begin{lemma}\lab{lemma:S}
We have
$$S=\{p\in\Sigma_s,\,/\, m(p)=0\}.$$
\end{lemma}

The proof of Lemma \ref{lemma:S} is postponed to Appendix \ref{sec:lemmaS}. Next, we define the following two  subsets of $\Sigma_s$
\be\lab{defAintAext}
A_{int}=\{p\in\Sigma_s\,/\, m(p)<0\},\, A_{ext}=\{p\in\Sigma_s\,/\, m(p)>0\}.
\ee
Note in view of Lemma \ref{lemma:S} that 
\be\lab{union}
\Sigma_s=S\sqcup  A_{int}\sqcup A_{ext}.
\ee

\begin{remark}
In the flat case, the picture is the following:
\begin{enumerate}
\item The null geodesics\footnote{which are straight lines in this case} $\ga_\o$ span the light cone from $(t,x)$. In particular, the null geodesics $\ga_\o$ do not intersect except at $(t,x)$.

\item $S$ is the intersection\footnote{$S$ is a sphere in this case} of the forward light cone from $(t,x)$ with $\{s\}\times\R^3$.

\item $A_{int}$ and $A_{ext}$ correspond respectively to the interior and the exterior of $S$.
\end{enumerate}
Note that we do not need to prove these statements in our case. This is fortunate since these statements - while probably true in our general setting - would be delicate to establish (see for instance \cite{KRradius} for a proof of (1) on a space-time $(\MM,\gg)$ with limited regularity).
\end{remark}

Next, we introduce some further notations. First, we denote by $m_0$ the value of $m$ at $(s,y)$, i.e.
\be\lab{def:m0}
m_0=\max_{\o\in\S}(u(s,y,\o)-u(t,x,\o)).
\ee
We also denote by $\o_0$ an angle in $\S$ where the maximum in \eqref{def:m0} is achieved, i.e.
\be\lab{def:om0}
m_0=u(s,y,\o_0)-u(t,x,\o_0).
\ee

\begin{remark}
In the flat case, $\o_0$ is unique and corresponds to the angle of the projection of $(s,y)$ on $S$. Again, while this may be also true in our general setting, we do not need to prove this statement in our case.
\end{remark}

Note that if $(s,y)\in A_{ext}$, the function $u(s,y,\o)-u(t,x,\o)$ may change sign as $\o$ varies on $\S$.
We define
\be\lab{eo1}
D=\{\o\in\S\,/\, u(t,x,\o)=u(s,y,\o)\}.
\ee
The following lemma gives a precise description of $D$. 
\begin{lemma}\lab{lemma:D}
Let $(s,y)\in A_{ext}$. Let $D$ defined as in \eqref{eo1}. Let $(\theta, \varphi)$ denote the spherical coordinates with axis $\o_0$. Then, there exists a $C^1$ $2\pi$-periodic function  
$$\theta_1:[0,2\pi)\rightarrow (0,\pi)$$
such that in the coordinate system $(\theta, \varphi)$, $D$ is parametrized by
$$D=\{\theta=\theta_1(\varphi),\, 0\leq\varphi<2\pi\}.$$
\end{lemma}

The proof of Lemma \ref{lemma:D} is postponed to Appendix \ref{sec:lemmaD}. 

\begin{remark}
In the flat case, recall that $u(t,x,\o)=-t+x\c\o$. In this case, one easily checks that $D$ is a circle of axis $\o_0$ 
on the sphere $S$ which is generated by the tangents to $S$ through $y$ (see figure \ref{figbis}).
\end{remark}

Let $\o\in\S$. According to Lemma \ref{lemma:D}, the great half circle on $\S$ originating at $\o_0$ and containing $\o$ intersects  $D$ at a fixed point $\o_1$. Let $\th$ and $\th_1$ respectively denote  the positive angles between $\o_0$ and $\o$  (resp. $\o_0$ and $\o_1$). \\

In order to obtain a lower bound for $|\phi|$, we will argue differently according to whether $(s,y)$ belongs to the region $S$, $A_{int}$ or $A_{ext}$.
\begin{lemma}[Key lemma]\lab{lemma:key}
$|\phi|$ satisfies the following lower bounds
\begin{enumerate}
\item If $(s,y)\in S$, we have
\be\lab{keylowb1}
|\phi(t,x,s,y,\o)|\geq \frac{1}{4}|t-s||\o-\o_0|^2.
\ee

\item If $(s,y)\in A_{int}$, we have
\be\lab{keylowb2}
|\phi(t,x,s,y,\o)|\geq \frac{1}{8}|t-s||\o-\o_0|^2.
\ee

\item If $(s,y)\in A_{ext}$ and $\theta_1\leq\theta\leq \pi$, we have
\be\lab{keylowb3}
|\phi(t,x,s,y,\o)|\geq \frac{1}{4}|t-s||\o-\o_1|^2.
\ee

\item If $(s,y)\in A_{ext}$ and $0\leq\theta\leq\theta_1$, we have
\be\lab{keylowb4}
|\phi(t,x,s,y,\o)| \gtrsim \sqrt{\frac{1-\cos(\theta-\theta_1)}{1-\cos(\theta_1)}}m_0
\ee
\end{enumerate}
\end{lemma}

The proof of Lemma \ref{lemma:key} is postponed to section \ref{sec:lowerb}. 

\begin{figure}[t]
\vspace{-4cm}
\begin{center}
\hspace{-4.0005cm}\includegraphics[width=20cm, height=13cm]{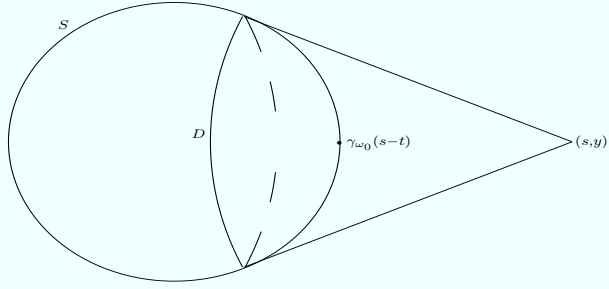}
\vspace{-6.5cm}
\caption{Representation of $D$ in the flat case}\lab{figbis}
\end{center}
\end{figure}

\begin{remark}\lab{reminiscentoverlap}
The proof of Lemma \ref{lemma:key} is inspired by the overlap estimates for wave packets derived in \cite{Sm} and \cite{SmTa} in the context of Strichartz estimates respectively for $C^{1,1}$ and $H^{2+\ep}$ metrics. Note however that the estimates in these papers rely heavily on a direct comparison of various quantities with the corresponding ones in the flat case. Such direct comparisons do not hold in our framework. Here, the closeness to the flat case manifests itself in the small constant $\ep$ in the right-hand side of \eqref{regb}, \eqref{regpoN} and \eqref{ad1}, and in the existence of the global coordinates systems of section \ref{sec:assumptioncoord}.
\end{remark}

\subsection{Proof of Proposition \ref{prop:estkernel}}

Recall that we need to show that the kernel $K$ defined in \eqref{defkernel} satisfies the upper bound \eqref{estkernel}. To this end, we will use the estimate \eqref{ibpkernel} for $K$ together with the estimates provided by Lemma \ref{lemma:key}. We argue differently according according to whether $(s,y)$ belongs to $S$, $A_{int}$ or $A_{ext}$. 

\subsubsection{The case $(s,y)\in S$}

If $(s,y)$ belongs to $S$, we have the lower bound \eqref{keylowb1} for $|\phi|$
$$|\phi(t,x,s,y,\o)|\geq \frac{1}{4}|t-s||\o-\o_0|^2,$$
where $\o_0\in\S$ is an angle satisfying \eqref{def:om0}. Then, we deduce
$$2^j\left|\phi\left(\frac{t}{2^j},\frac{x}{2^j},\frac{s}{2^j},\frac{y}{2^j},\o\right)\right|\geq \frac{1}{4}|t-s||\o-\o_0|^2.$$
Together with \eqref{ibpkernel}, this yields
$$|K(t,x,s,y)|\les \int_{\S}\frac{d\o}{1+|t-s|^2|\o-\o_0|^4}.$$
Using the spherical coordinates $(\theta, \varphi)$ with axis $\o_0$, we obtain
$$|K(t,x,s,y)|\les \int_0^{\pi}\frac{\sin(\theta)d\theta}{1+|t-s|^2(1-\cos(\theta))^2}.$$
Performing the change of variables
$$z=|t-s|(1-\cos(\theta))$$
we obtain
$$|K(t,x,s,y)|\les \frac{1}{|t-s|}\int_0^{+\infty}\frac{dz}{1+z^2}.$$
This implies
\be\lab{ok}
|K(t,x,s,y)|\les \frac{1}{|t-s|},\,\,\,\forall (t,x)\in 2^j\MM,\,\,\, \forall \left(\frac{s}{2^j},\frac{y}{2^j}\right)\in S 
\ee
which is the desired estimate.

\subsubsection{The case $(s,y)\in A_{int}$}

If $(s,y)$ belongs to $A_{int}$, we have the lower bound \eqref{keylowb2} for $|\phi|$
$$|\phi(t,x,s,y,\o)|\geq \frac{1}{8}|t-s||\o-\o_0|^2,$$
where $\o_0\in\S$ is an angle satisfying \eqref{def:om0}. Arguing as in the previous case, we obtain
\be\lab{ok1}
|K(t,x,s,y)|\les \frac{1}{|t-s|},\,\,\,\forall (t,x)\in 2^j\MM,\,\,\, \forall \left(\frac{s}{2^j},\frac{y}{2^j}\right)\in A_{int}. 
\ee

\subsubsection{The case $(s,y)\in A_{ext}$}

If $(s,y)$ belongs to $A_{ext}$, recall that $\o_1$ is in $D$ such that $\o$, $\o_1$ and $\o_0$ are on the same half great  circle of $\S$, and that $\th$ and $\th_1$ denote respectively the positive angles between $\o_0$ and $\o$  (resp. $\o_0$ and $\o_1$).\\

\noindent{\bf The case $\theta_1\leq \theta\leq \pi$.} If $\theta_1\leq \theta\leq \pi$, we have the lower bound \eqref{keylowb3} for $|\phi|$
$$|\phi(t,x,s,y,\o)| \geq  \frac{1}{2}|t-s||\o-\o_1|^2.$$
Then, we deduce
$$2^j\left|\phi\left(\frac{t}{2^j},\frac{x}{2^j},\frac{s}{2^j},\frac{y}{2^j},\o\right)\right|\geq \frac{1}{2}|t-s||\o-\o_1|^2.$$
Together with \eqref{ibpkernel}, this yields
$$|K(t,x,s,y)|\les \int_{\S}\frac{d\o}{1+|t-s|^2|\o-\o_1|^4}.$$
Using the spherical coordinates $(\theta, \varphi)$ with axis $\o_0$, we parametrize $\S$ by $(\theta-\theta_1(\varphi), \varphi)$ where $\varphi\to\theta_1(\varphi)$ is defined in Lemma \ref{lemma:D}. We obtain
$$|K(t,x,s,y)|\les \int_0^{2\pi}\int_{\theta_1(\varphi)}^{\pi}\frac{\sin(\theta-\theta_1(\varphi))}{1+|t-s|^2(1-\cos(\theta-\theta_1(\varphi)))^2}d\theta d\varphi$$
and thus
$$|K(t,x,s,y)|\les \int_0^{\pi}\frac{\sin(\theta)}{1+|t-s|^2(1-\cos(\theta))^2}d\theta.$$
Performing the change of variable
$$z=|t-s|(1-\cos(\theta))$$
we obtain
$$|K(t,x,s,y)|\les \frac{1}{|t-s|}\int_0^{+\infty}\frac{dz}{1+z^2}.$$
This implies
\be\lab{ok3}
|K(t,x,s,y)|\les \frac{1}{|t-s|},\,\,\,\forall (t,x)\in 2^j\MM,\,\,\, \forall \left(\frac{s}{2^j},\frac{y}{2^j}\right)\in A_{ext}\textrm{ with }\theta_1\leq \theta\leq \pi 
\ee
which is the desired estimate.

\vspace{0.5cm} 

\noindent{\bf The case $0\leq\theta\leq\theta_1$.} Finally, if $0\leq\theta\leq\theta_1$, we have the lower bound \eqref{keylowb4} for $|\phi|$
$$|\phi(t,x,s,y,\o)| \gtrsim \sqrt{\frac{1-\cos(\theta-\theta_1)}{1-\cos(\theta_1)}}m_0.$$
We then deduce
$$2^j\left|\phi\left(\frac{t}{2^j},\frac{x}{2^j},\frac{s}{2^j},\frac{y}{2^j},\o\right)\right|\gtrsim 2^j\sqrt{\frac{1-\cos(\theta-\theta_1)}{1-\cos(\theta_1)}}m_j$$
where $m_j$ is defined as 
$$m_j=\max_{\o\in\S}\left(u\left(\frac{s}{2^j},\frac{y}{2^j},\o\right)-u\left(\frac{t}{2^j},\frac{x}{2^j},\o\right)\right).$$
Together with \eqref{ibpkernel}, this yields
$$|K(t,x,s,y)|\les \int_0^{\theta_1}\frac{\sin(\theta)}{1+2^{2j}m_j^2\frac{1-\cos(\theta-\theta_1)}{1-\cos(\theta_1)}}d\theta.$$
Performing the change of variable
$$z=2^jm_j\sqrt{\frac{1-\cos(\theta-\theta_1)}{1-\cos(\theta_1)}}$$
and using \eqref{oo7} and the fact that 
$$\frac{\sin(\theta)}{\sqrt{1-\cos(\theta)}}\les 1\textrm{ and }\sin(\theta)\les\sin(\theta_1)\textrm{ on }0\leq\theta\leq\theta_1\leq\frac{\pi}{2}+O(\ep),$$
we obtain
\be\lab{oo16}
|K(t,x,s,y)|\les \frac{\sin(\theta_1)\sqrt{1-\cos(\theta_1)}}{2^jm_j}\int_0^{+\infty}\frac{dz}{1+z^2}\les  \frac{\sin(\theta_1)\sqrt{1-\cos(\theta_1)}}{2^jm_j}.
\ee
Now, in view of \eqref{oo6}, we have
$$\sin(\theta_1)\sqrt{1-\cos(\theta_1)}\les\frac{1}{1+\frac{|t-s|}{2^jm_j}}.$$
Together with \eqref{oo16}, we obtain
$$|K(t,x,s,y)|\les  \frac{1}{2^jm_j+|t-s|}$$
which implies
\be\lab{ok4}
|K(t,x,s,y)|\les \frac{1}{|t-s|},\,\,\,\forall (t,x)\in 2^j\MM,\,\,\, \forall \left(\frac{s}{2^j},\frac{y}{2^j}\right)\in A_{ext}\textrm{ with }0\leq\theta\leq\theta_1. 
\ee

Finally, \eqref{union}, \eqref{ok}, \eqref{ok1}, \eqref{ok3} and \eqref{ok4} yield \eqref{estkernel} which concludes the proof of Proposition \ref{prop:estkernel}.

\section{Proof of Lemma \ref{lemma:key} (Lower bound for $|\phi|$)}\lab{sec:lowerb}

\subsection{A lower bound for $|\phi|$ when $(s,y)\in S$ (proof of \eqref{keylowb1})} 

In view of the definition of $S$, there is $\o_0\in\S$ such that
$$(s,y)=\gamma_{\o_0}(s-t).$$
In view of \eqref{idgeod}, this yields
\bea\lab{eee}
u(s,y,\o)-u(t,x,\o) &=& u(\gamma_{\o_0}(s-t),\o)-u(t,x,\o)\\
\nn&=& \int_0^{s-t}\gg(\dd u, \gamma_{\o_0}'(\sigma)d\sigma\\
\nn&=&  \int_0^{s-t}b^{-1}(\gamma_{\o_0}(\sigma),\o)\gg(L(\gamma_{\o_0}(\sigma),\o), L(\gamma_{\o_0}(\sigma),\o_0)d\sigma\\
\nn&=&  -\frac{1}{2}\int_0^{s-t}b^{-1}(\gamma_{\o_0}(\sigma),\o)|N(\gamma_{\o_0}(\sigma),\o)-N(\gamma_{\o_0}(\sigma),\o_0)|^2d\sigma\\
\nn&\leq& -\frac{1}{4}|t-s||\o-\o_0|^2,
\eea 
where we used in the last inequality the estimates \eqref{regb} and \eqref{ad1} with $\ep>0$ small enough. \eqref{eee} implies for all $(s,y)\in S$ and all $\o\in\S$
\be\lab{eee1}
|\phi(t,x,s,y,\o)|\geq \frac{1}{4}|t-s||\o-\o_0|^2,
\ee
which is the desired estimate \eqref{keylowb1}.

\subsection{A lower bound for $|\phi|$ when $(s,y)\in A_{int}$ (proof of \eqref{keylowb2})}\lab{sec:cccc}

Recall that $m_0<0$ since $(s,y)\in A_{int}$. Let $\o_0\in\S$ an angle satisfying \eqref{def:om0}. Then, we have in particular
$$\po u(s,y,\o_0)=\po u(t,x,\o_0).$$
Together with \eqref{identities}, this yields
\be\lab{ie}
\po u(s,y,\o_0)=\po u(\gamma_{\o_0}(s-t),\o_0). 
\ee

In view of the assumption \eqref{regpoN}, the $2\times 2$ matrix
$$\gg(\po N, \po N)$$
is invertible. We define the vector of $T\S$ $a$ as
\be\lab{iebis}
a=\gg(\po N, \po N)^{-1}\po b.
\ee
Note in view of \eqref{regb} and \eqref{regpoN} that $a$ satisfies the estimate
\be\lab{iequatre}
\norm{a}_{L^\infty}\les\ep.
\ee
For $\sigma\in\R$, let us consider the curve $\mu(\sigma)$ defined by 
\be\lab{ieter}
\left\{\ba{lll}
\mu'(\sigma)&=&b(\mu(\sigma),\o_0)N(\mu(\sigma),\o_0)+a(\mu(\sigma),\o_0)\c\po N(\mu(\sigma),\o_0), \\
\mu(0)&=&\gamma_{\o_0}(s-t).
\ea\right.
\ee

\begin{remark}
In the flat case, the curve $\mu$ is simply the segment of straight line between $\gamma_{\o_0}(s-t)$ and $(s,y)$ (see figure \ref{fig4}).
\end{remark}
 \begin{figure}[b]
\vspace{-4.5cm}
\begin{center}
\hspace{-4.0005cm}\includegraphics[width=20cm, height=15cm]{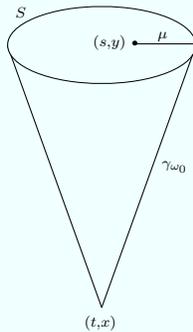}
\vspace{-6cm}
\caption{The case $(s,y)\in A_{int}$ in the flat case}\lab{fig4}
\end{center}
\end{figure}

\begin{lemma}\lab{lemma:appendixc}
Let $\mu$ the curve defined in \eqref{ieter}. Then, we have 
\be\lab{eq:appendixc}
(s,y)=\mu(m_0).
\ee
\end{lemma}

The proof of Lemma \ref{lemma:appendixc} is postponed to Appendix \ref{sec:appendixc}. \eqref{eq:appendixc}  yields
\bee
 u(s,y,\o)-u(t,x,\o) &=& u(\mu(m_0),\o)-u(\mu(0),\o)+u(\gamma_{\o_0}(s-t),\o)-u(t,x,\o)\\
&=& \int_0^{m_0}\gg(\nabla u(\mu(\sigma),\o), \mu'(\sigma))d\sigma +\int_t^s\gg(\dd u(\gamma_{\o_0}(\sigma),\o), \gamma_{\o_o}'(\sigma))d\sigma\\
&=& \int_0^{m_0} b^{-1}(\mu(\sigma),\o)\Big(b(\mu(\sigma),\o_0)\gg(N(\mu(\sigma),\o), N(\mu(\sigma),\o_0))\\
&&+a(\mu(\sigma),\o_0)\c\gg(\po N(\mu(\sigma),\o_0), N(\mu(\sigma),\o))\Big) \\
&&+\int_t^sb^{-1}(\gamma_{\o_0}(\sigma),\o)b^{-1}(\gamma_{\o_0}(\sigma),\o_0)\gg(L(\gamma_{\o_0}(\sigma),\o),L(\gamma_{\o_0}(\sigma),\o_0)).
\eee
We obtain
\bee
 u(s,y,\o)-u(t,x,\o) &= & -\int_0^{|m_0|}\Big(\gg(N(\mu(\sigma),\o),N(\mu(\sigma),\o_0))+O(\ep)\Big)d\sigma\\
&& -\frac{1}{2}\int_t^s|N(\gamma_{\o_0}(\sigma),\o)-N(\gamma_{\o_0}(\sigma),\o_0)|^2(1+O(\ep))d\sigma
\eee
where we used the estimates \eqref{regb} and \eqref{regpoN}, the identity \eqref{identities1}, the fact that $s>t$, and the fact that $m_0<0$ since $(s,y)\in A_{int}$. Together with \eqref{ad1}, this yields
\be\lab{ie4}
u(s,y,\o)-u(t,x,\o) =  -\frac{1}{2}|t-s||\o-\o_0|^2(1+O(\ep))-|m_0|(\o\cdot\o_0+O(\ep)).
\ee
In particular we deduce for $\ep>0$ small enough
\be\lab{ie5}
u(s,y,\o)-u(t,x,\o) \leq  -\frac{1}{4}|t-s||\o-\o_0|^2\textrm{ for all }\o\textrm{ such that }\o\cdot\o_0\geq \frac{1}{4}.
\ee

Since $\o_0$ is an angle where the maximum in the definition \eqref{def:m0} of $m_0$ is attained, we have for all $\o\in\S$, in view of \eqref{ie3} and \eqref{ie4}, and the fact that $m_0<0$ 
$$-|m_0|\geq -\frac{1}{2}|t-s||\o-\o_0|^2(1+O(\ep))-|m_0|(\o\cdot\o_0+O(\ep)).$$
This yields
$$|m_0|\leq |t-s|(1+O(\ep)).$$
Injecting back in \eqref{ie4}, we obtain for $\ep>0$ small enough
$$u(s,y,\o)-u(t,x,\o) \leq  -\frac{1}{2}|t-s|\textrm{ for all }\o\textrm{ such that }\o\cdot\o_0\leq \frac{1}{4}.$$
Together with \eqref{ie5}, we finally obtain for all $\o\in\S$
\be\lab{ie6}
|\phi(t,x,s,y,\o)|\geq \frac{1}{8}|t-s||\o-\o_0|^2
\ee
which is the desired estimate \eqref{keylowb2}.

\subsection{A lower bound for $|\phi|$ when $(s,y)\in A_{ext}$ (proof of \eqref{keylowb3} \eqref{keylowb4})}

Let $m_0$ defined in \eqref{def:m0}. Let $\o_0\in\S$ an angle satisfying \eqref{def:om0}. Note that $m_0>0$ since $(s,y)\in A_{ext}$. In particular, proceeding as in section \ref{sec:cccc}, we obtain the following analog of \eqref{ie4}
\be\lab{eo}
u(s,y,\o)-u(t,x,\o) =  -\frac{1}{2}|t-s||\o-\o_0|^2(1+O(\ep))+m_0(\o\cdot\o_0+O(\ep)).
\ee
Recall the definition \eqref{eo1} of the set $D$
$$D=\{\o\in\S\,/\, u(t,x,\o)=u(s,y,\o)\}.$$
In view of \eqref{eo}, if $\o_1\in D$, then
\be\lab{oe2}
1-\o_1\c\o_0=\frac{1+O(\ep)}{1+\frac{|t-s|}{m_0}(1+O(\ep))}.
\ee

Next, we consider $\o_1\in D$. In view of the assumption \eqref{regpoN}, the $2\times 2$ matrix
$$\gg(\po N, \po N)$$
is invertible. We define the vector of $T_{\o_1}\S$ $a_1$ as
\be\lab{oobis}
a_1=\gg(\po N, \po N)^{-1}\left(\po u(s,y,\o_1)-\po u(\gamma_{\o_1}(s-t),\o_1)\right).
\ee
Let us consider the curve $\eta(\sigma)$ defined by
\be\lab{ooter}
\left\{\ba{lll}
\eta'(\sigma)&=&b(\eta(\sigma),\o_1)a_1\c\po N(\eta(\sigma),\o_1),\\ 
\eta(0)&=&\gamma_{\o_1}(s-t).
\ea\right.
\ee

\begin{remark}
In the flat case, the curve $\eta$ is simply the segment of straight line between $\gamma_{\o_1}(s-t)$ and $(s,y)$ (see figure \ref{fig5}).
\end{remark}
 \begin{figure}[b]
\vspace{-3.5cm}
\begin{center}
\hspace{-4.0005cm}\includegraphics[width=20cm, height=15cm]{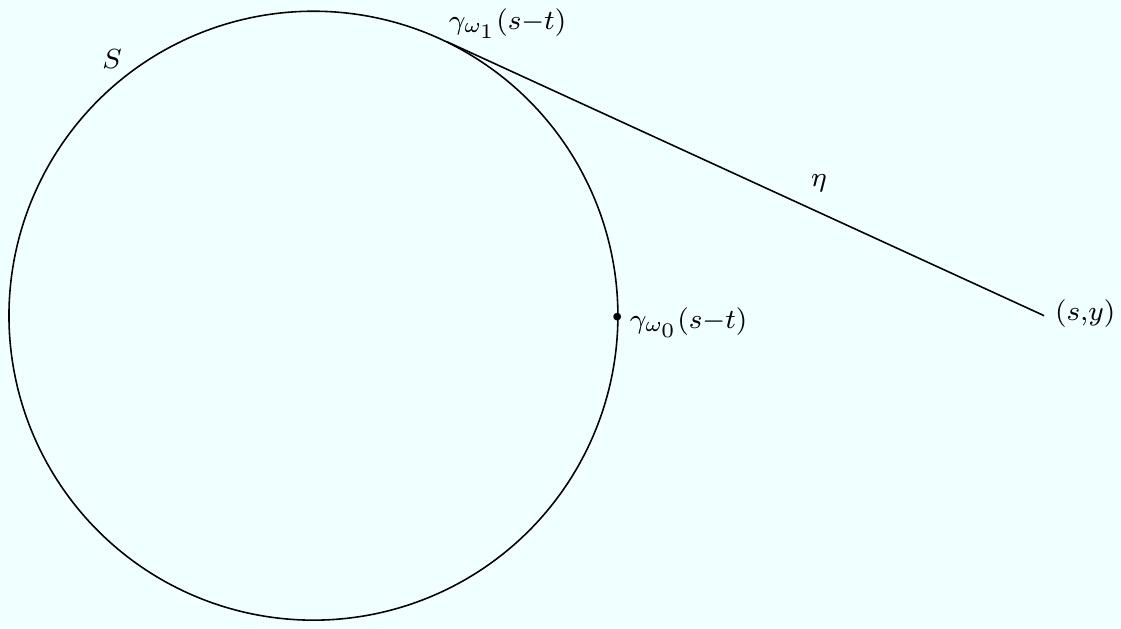}
\vspace{-7.5cm}
\caption{The case $(s,y)\in A_{ext}$ in the flat case}\lab{fig5}
\end{center}
\end{figure}

\begin{lemma}\lab{lemma:appendixd}
Let $\eta$ the curve defined in \eqref{ooter}. Then, we have 
\be\lab{eq:appendixd}
(s,y)=\eta(1).
\ee
\end{lemma}

The proof of Lemma \ref{lemma:appendixd} is postponed to Appendix \ref{sec:appendixd}. \eqref{eq:appendixd}  yields
\bee
 u(s,y,\o)-u(t,x,\o) &=& u(\eta(1),\o)-u(\eta(0),\o)+u(\gamma_{\o_1}(s-t),\o)-u(t,x,\o)\\
&=& \int_0^1\gg(\nabla u(\eta(\sigma),\o), \eta'(\sigma))d\sigma +\int_t^s\gg(\dd u(\gamma_{\o_1}(\sigma),\o), \gamma_{\o_o}'(\sigma))d\sigma\\
&=& \int_0^1 b^{-1}(\eta(\sigma),\o)b(\eta(\sigma),\o_1)\gg(N(\eta(\sigma),\o), a_1\c\po N(\eta(\sigma),\o_1))\\
&&+\int_t^sb^{-1}(\gamma_{\o_1}(\sigma),\o)b^{-1}(\gamma_{\o_1}(\sigma),\o_1)\gg(L(\gamma_{\o_1}(\sigma),\o),L(\gamma_{\o_1}(\sigma),\o_1)).
\eee
We obtain
\bee
 u(s,y,\o)-u(t,x,\o) &= & \int_0^1\gg(N(\eta(\sigma),\o)-N(\eta(\sigma),\o_1), a_1\c\po N(\eta(\sigma),\o_1))(1+O(\ep))\\
&& -\frac{1}{2}\int_t^s|N(\gamma_{\o_1}(\sigma),\o)-N(\gamma_{\o_1}(\sigma),\o_1)|^2(1+O(\ep))d\sigma
\eee
where we used the estimates \eqref{regb} and the identity \eqref{identities1}. Together with the assumptions \eqref{regpoN} and \eqref{ad1}, this yields
\bea\lab{oo3}
u(s,y,\o)-u(t,x,\o) &=&  -\frac{1}{2}|t-s||\o-\o_1|^2(1+O(\ep))\\
\nn &+& (\o-\o_1)\c(\po u(s,y,\o_1)-\po u(\gamma_{\o_1}(s-t),\o_1))(1+O(\ep)).
\eea

We introduce the notation $v_0$ for the following vector in $\R^3$. 
\be\lab{oo4}
v_0=\po u(s,y,\o_1)-\po u(\gamma_{\o_1}(s-t),\o_1).
\ee
Recall that $\o_0\in\S$ is an angle satisfying \eqref{def:om0}. In view of \eqref{oo3}, this yields
$$m_0= -\frac{1}{2}|t-s||\o_0-\o_1|^2(1+O(\ep))+(\o_0-\o_1)\c v_0(1+O(\ep)).$$
We deduce
\be\lab{oo5}
|v_0|=\frac{1}{|\o_0-\o_1|\cos(\a_1)}\left(m_0+\frac{1}{2}|t-s||\o_0-\o_1|^2(1+O(\ep))\right)(1+O(\ep)),
\ee
where $\a_1$ denotes the angle between $v_0$ and $\o_1-\o_0$. Let us denote by $\theta_1$ the angle between $\o_0$ and $\o_1$. In view of \eqref{oe2}, we have
\be\lab{oo6}
1-\cos(\theta_1)=\frac{1+O(\ep)}{1+\frac{|t-s|}{m_0}(1+O(\ep))}
\ee
and we deduce in particular
\be\lab{oo7}
0<\theta_1\leq\frac{\pi}{2}+O(\ep).
\ee
Note also in view of the definition \eqref{oo4} that $v_0$ belongs to $T_{\o_1}\S$ so that 
\be\lab{oo8}
v_0\c\o_1=0.
\ee
Simple considerations on angles imply\footnote{Let $\varphi_1$ the angle defined on figure \ref{fig6}. Then $2\varphi_1+\th_1=\pi$, and $\varphi_1+\alpha_1=\frac{\pi}{2}$ in view of \eqref{oo8}. Hence $\th_1=2\alpha_1$} (see figure \ref{fig6})
$$\alpha_1=\frac{\theta_1}{2}.$$
 \begin{figure}[t]
\vspace{-4cm}
\begin{center}
\hspace{-4.0005cm}\includegraphics[width=20cm, height=14.5cm]{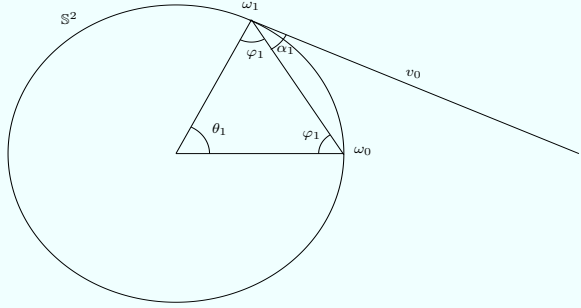}
\vspace{-8cm}
\caption{Definition of the angles $\th_1$ and $\alpha_1$}\lab{fig6}
\end{center}
\end{figure}
Together with \eqref{oo5}, this yields
\be\lab{oo9}
|v_0|=\frac{1}{|\o_0-\o_1|\cos\left(\frac{\theta_1}{2}\right)}\left(m_0+\frac{1}{2}|t-s||\o_0-\o_1|^2(1+O(\ep))\right)(1+O(\ep)).
\ee

Let $\o\in\S$. According to Lemma \ref{lemma:D}, the half great circle on $\S$ originating at $\o_0$ and containing $\o$ intersects  $D$ at a unique point $\o_1$. Let $\th$ denote  the positive angle between $\o_0$ and $\o$ and let $\alpha$ denote the angle between $v_0$ and $\o_1-\o$. In view of \eqref{oo8}, 
simple considerations on angles imply\footnote{Let $\varphi$ the angle defined on figure \ref{fig7}. Then $2\varphi+|\th_1-\th|=\pi$, and $\varphi+\alpha=\frac{\pi}{2}$ in view of \eqref{oo8}. Hence $|\th_1-\th|=2\alpha$} (see figure \ref{fig7})
$$\alpha=\frac{|\theta_1-\th|}{2}.$$
 \begin{figure}[t]
\vspace{-3.8cm}
\begin{center}
\hspace{-4.0005cm}\includegraphics[width=20cm, height=15cm]{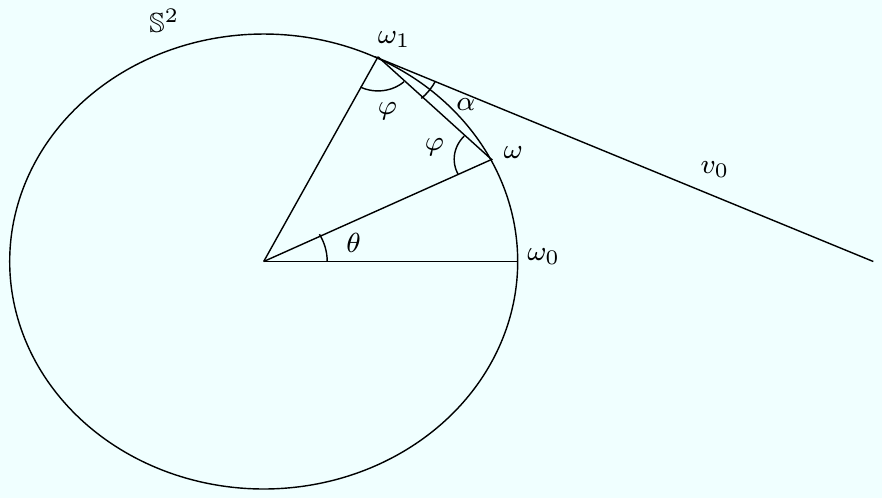}
\vspace{-8.3cm}
\caption{Definition of the angles $\th$ and $\alpha$}\lab{fig7}
\end{center}
\end{figure}
Together with \eqref{oo3}, the definition of $\phi$, and the definition \eqref{oo4}, we have either
\be\lab{oo10}
\phi(t,x,s,y,\o) =  -\frac{1}{2}|t-s||\o-\o_1|^2(1+O(\ep))+|\o-\o_1||v_0|\cos\left(\frac{\theta-\theta_1}{2}\right)(1+O(\ep)),
\ee
if $0\leq \theta\leq \theta_1$, or 
\be\lab{oo11}
\phi(t,x,s,y,\o) =  -\frac{1}{2}|t-s||\o-\o_1|^2(1+O(\ep))-|\o-\o_1||v_0|\cos\left(\frac{\theta-\theta_1}{2}\right)(1+O(\ep)),
\ee
if $\theta_1\leq\theta\leq\pi$, where we have used the fact that (see figure \ref{fig7})
$$(\o-\o_1)\cdot v_0\geq 0\textrm{ if }0\leq\th\leq\th_1\textrm{ and }(\o-\o_1)\cdot v_0<0\textrm{ if }\th_1<\th\leq\pi.$$
We consider the two cases in the next two sections.

\subsubsection{The case $\theta_1\leq\theta\leq\pi$}

We are in the case \eqref{oo11}, so that we have for $\ep>0$ small enough
$$\phi(t,x,s,y,\o) \leq  -\frac{1}{4}|t-s||\o-\o_1|^2.$$
In particular, we obtain 
\be\lab{oo12}
|\phi(t,x,s,y,\o)| \geq  \frac{1}{4}|t-s||\o-\o_1|^2
\ee
which is the desired estimate \eqref{keylowb3}. 

\subsubsection{The case $0\leq \theta\leq \theta_1$}

We are in the case \eqref{oo10}, which together with \eqref{oo9} yields
\bea\lab{oo13}
 \phi(t,x,s,y,\o) &=& \frac{|\o-\o_1|}{|\o_0-\o_1|}\frac{\cos\left(\frac{\theta-\theta_1}{2}\right)}{\cos\left(\frac{\theta_1}{2}\right)}m_0
+\frac{1}{2}|t-s||\o-\o_1| A(\o)(1+O(\ep))\\
\nn&&+|t-s||\o-\o_1|^2O(\ep) 
\eea
where $A$ is given by
\be\lab{oo14}
A(\o)=-|\o-\o_1|+\frac{\cos\left(\frac{\theta-\theta_1}{2}\right)}{\cos\left(\frac{\theta_1}{2}\right)}|\o_0-\o_1|.
\ee
Since $\theta$ is the angle between $\o$ and $\o_0$, and $\theta_1$ is the angle between $\o_1$ and $\o_0$, we have
\be\lab{oo14bis}
|\o-\o_1|=\sqrt{2}\sqrt{1-\cos(\theta-\theta_1)},\, |\o_1-\o_0|=\sqrt{2}\sqrt{1-\cos(\theta_1)}.
\ee
Together with \eqref{oo14}, we obtain
$$A(\o)=\sqrt{2}\sqrt{1+\cos(\theta_1)}\left(\frac{\sqrt{1-\cos(\theta_1)}}{\sqrt{1+\cos(\theta_1)}}-\frac{\sqrt{1-\cos(\theta-\theta_1)}}{\sqrt{1+\cos(\theta-\theta_1)}}\right)$$
which yields
$$A(\o)\geq 0\textrm{ for all }0\leq\theta\leq\theta_1.$$
Together with \eqref{oo13}, we obtain
$$\phi(t,x,s,y,\o) \geq \frac{|\o-\o_1|}{|\o_0-\o_1|}\cos\left(\frac{\theta-\theta_1}{2}\right)m_0+|t-s||\o-\o_1|^2O(\ep).$$
In view of the fact that from \eqref{oo7} we have
$$0\leq \theta\leq \theta_1\leq\frac{\pi}{2}+O(\ep),$$
we deduce
\be\lab{presque}
\phi(t,x,s,y,\o) \geq \frac{\sqrt{2}}{2}\frac{|\o-\o_1|}{|\o_0-\o_1|}m_0+|t-s||\o-\o_1|^2O(\ep).
\ee
Evaluating \eqref{eo} at $\o=\o_1$ and using the fact that $\o_1\in D$ so that the left-hand side of \eqref{eo} vanishes, we obtain
$$|t-s||\o_1-\o_0|^2\les m_0$$
which together with \eqref{presque} and the fact that
$$|\o-\o_1|\leq |\o_0-\o_1|$$ 
yields for $\ep>0$ small enough
$$\phi(t,x,s,y,\o) \gtrsim \frac{|\o-\o_1|}{|\o_0-\o_1|}m_0.$$
In view of \eqref{oo14bis}, we deduce
\be\lab{oo15}
\phi(t,x,s,y,\o) \gtrsim \sqrt{\frac{1-\cos(\theta-\theta_1)}{1-\cos(\theta_1)}}m_0
\ee
which is the desired estimate \eqref{keylowb4}.\\

Finally, in view of \eqref{eee1}, \eqref{ie6}, \eqref{oo12} and \eqref{oo15}, we have obtained the desired estimates \eqref{keylowb1}, \eqref{keylowb2}, \eqref{keylowb3} and \eqref{keylowb4}. This concludes the proof of Lemma \ref{lemma:key}. 

\appendix

\section{Proof of Lemma \ref{lemma:S}}\lab{sec:lemmaS}

If $p\in S$, then, there is $\o_0$ such that 
$$p=\gamma_{\o_0}(s-t).$$
In view of \eqref{idgeod}, this yields
\bea\lab{ddd}
u(p,\o)-u(t,x,\o) &=& u(\gamma_{\o_0}(s-t),\o)-u(t,x,\o)\\
\nn&=& \int_0^{s-t}\gg(\dd u, \gamma_{\o_0}'(\sigma))d\sigma\\
\nn&=&  \int_0^{s-t}b^{-1}(\gamma_{\o_0}(\sigma),\o)\gg(L(\gamma_{\o_0}(\sigma),\o), L(\gamma_{\o_0}(\sigma),\o_0)d\sigma\\
\nn&\leq& 0
\eea 
where we used in the last inequality the fact that the scalar product of 2 null vectors is negative
$$\gg(L(\gamma_{\o_0}(\sigma),\o), L(\gamma_{\o_0}(\sigma),\o_0)\leq 0.$$
Arguing as in \eqref{ddd} in the special case $\o=\o_0$, we obtain
\bea\lab{ddd1}
u(p,\o_0)-u(t,x,\o_0) &=& u(\gamma_{\o_0}(s-t),\o_0)-u(t,x,\o_0)\\
\nn&=& \int_0^{s-t}\gg(\dd u, \gamma_{\o_0}'(\sigma))d\sigma\\
\nn&=&  \int_0^{s-t}b^{-1}(\gamma_{\o_0}(\sigma),\o_0)\gg(L(\gamma_{\o_0}(\sigma),\o_0), L(\gamma_{\o_0}(\sigma),\o_0))d\sigma\\
\nn& = & 0
\eea 
since $L(\gamma_{\o_0}(\sigma),\o_0)$ is null. In view of \eqref{ddd} and \eqref{ddd1}, we finally obtain
\be\lab{ddd2}
S\subset\{p\in\Sigma_s,\,/\, m(p)=0\}.
\ee

Conversely, let $p$ such that $m(p)=0$. Let $\o_0\in\S$ an angle where the max in the definition \eqref{def:m} of $m$ is attained. Then, we have at $\o=\o_0$:
\be\lab{ddd3}
u(p,\o_0)=u(t,x,\o_0),\, \po u(p,\o_0)=\po u(t,x,\o_0).
\ee
Also, in view of \eqref{identities}, we have
$$u(\gamma_{\o_0}(s-t),\o_0)=u(t,x,\o_0),\, \po u(\gamma_{\o_0}(s-t),\o_0)=\po u(t,x,\o_0)$$
which together with \eqref{ddd3} implies
$$u(p,\o_0)=u(\gamma_{\o_0}(s-t),\o_0),\, \po u(p,\o_0)=\po u(\gamma_{\o_0}(s-t),\o_0).$$
Since $u(s,.,\o_0), \po u(s,.,\o_0)$ forms a global coordinate system on $\Sigma_s$ in view of the assumption in section \ref{sec:assumptioncoord}, we deduce
$$p=\gamma_{\o_0}(s-t)\in S$$
and thus 
$$\{p\in\Sigma_s,\,/\, m(p)=0\}\subset S.$$
Together with \eqref{ddd2}, this concludes the proof of Lemma \ref{lemma:S}.

\section{Proof of Lemma \ref{lemma:D}}\lab{sec:lemmaD}

Let $(\theta, \varphi)$ denote the spherical coordinates with axis $\o_0$. Note from the definition \eqref{eo1} of $D$ and the definition of $\phi$ that $D$ is given by
\be\lab{dodo}
D=\{\o\in\S,\,/\,\phi(t,x,s,y,\o)=0\}.
\ee
Recall that 
$$\phi(t,x,s,y,\o_0)=m_0>0.$$
Also, we have from \eqref{eo}
$$\phi(t,x,s,y,-\o_0)=-2|t-s|^2(1+O(\ep))+m_0(-1+O(\ep))<0.$$
Thus, since $\phi$ is continuous, we deduce from the mean value theorem that 
\be\lab{dodobis}
\forall\varphi\in [0,2\pi),\textrm{ there exists at least one }\theta_1\in(0,\pi)\textrm{ such that }(\theta_1, \varphi)\in D.
\ee
Also, note in view of \eqref{oo12} and \eqref{oo15}, 
$$\forall\varphi\in [0,2\pi),\textrm{ there exists at most one }\theta_1\in(0,\pi)\textrm{ such that }(\theta_1, \varphi)\in D.$$
Together with \eqref{dodobis}, we deduce the existence of a $2\pi$-periodic function 
$$\theta_1:[0,2\pi)\rightarrow (0,\pi)$$
such that in the coordinate system $(\theta, \varphi)$, $D$ is parametrized by
$$D=\{\theta=\theta_1(\varphi),\, 0\leq\varphi<2\pi\}.$$

To conclude the proof of Lemma \ref{lemma:D}, it remains to prove that $\theta_1$ is $C^1$. Let $\o_1\in D$. Let $(\theta_1, \varphi_1)$ the coordinates of $\o_1$. By a slight abuse of notations, let us identify $\o_1$ with $(\theta_1, \varphi_1)$. Then, since 
$$\o\rightarrow \phi(t,x,s,y,\o)$$
is a $C^1$ function from our assumptions on $u$, and since 
$$\phi(t,x,s,y,\o_1)=0$$
in view of the fact that $\o_1$ belongs to $D$, we have
\be\lab{dodo1}
\partial_\theta\phi(t,x,s,y,\o_1)=\lim_{\theta\rightarrow\theta_1}\frac{\phi(t,x,s,y,\theta,\varphi_1)}{\theta-\theta_1}.
\ee
Now, \eqref{oo10} and \eqref{oo11} imply
$$\phi(t,x,s,y,\theta,\varphi_1)=|v_0|(1+O(\ep))(\theta-\theta_1)(1+o(1))\textrm{ as }\theta\rightarrow\theta_1$$
which together with \eqref{dodo1} yields
\be\lab{dodo2}
\partial_\theta\phi(t,x,s,y,\o_1)\neq 0.
\ee
Finally, in view of \eqref{dodo2} and the fact that $\phi$ is $C^1$, the implicit function theorem implies that $\theta_1$ is a $C^1$ function. This concludes the proof of Lemma \ref{lemma:D}.

\section{Proof of Lemma \ref{lemma:appendixc}}\lab{sec:appendixc}

Note that 
\bee
u(\mu(\sigma),\o_0)' &=& \gg(\dd u(\mu(\sigma),\o_0), \mu'(\sigma))\\
&=& 1
\eee
by the definition of $\dd u$, $\mu'$, the identity \eqref{identities1}, and the fact that $N$ is unitary. This implies
\be\lab{ie1}
u(\mu(\sigma),\o_0)=\sigma+u(\gamma_{\o_0}(t-s),\o_0)
\ee
Also, we have
\bea\lab{ie2}
\po u(\mu(\sigma),\o_0)' &=& \gg(\dd \po u(\mu(\sigma),\o_0), \mu'(\sigma))\\
\nn&=& \gg(-b^{-2}\po b N(\mu(\sigma),\o_0)+b^{-1}\po N(\mu(\sigma),\o_0), \mu'(\sigma))\\
\nn&=& 0,
\eea
where we used in the last inequality the definition \eqref{ieter} of $\mu'$ and the definition \eqref{iebis} of $a$. 

Recall the definition \eqref{def:m0} of $m_0$ 
\be\lab{ie3}
m_0=u(s,y,\o_0)-u(t,x,\o_0).
\ee
In view of \eqref{ie} and \eqref{ieter}-\eqref{ie3}, we have
$$u(\mu(m_0),\o_0)=u(s,y,\o_0),\, \po u(\mu(m_0),\o_0)=\po u(s,y,\o_0).$$
Since $u(s,.,\o_0), \po u(s,.,\o_0)$ forms a global coordinate system on $\Sigma_s$ in view of the assumption in section \ref{sec:assumptioncoord}, we deduce
$$(s,y)=\mu(m_0)$$
which is the desired estimate. This concludes the proof of Lemma \ref{lemma:appendixc}.

\section{Proof of Lemma \ref{lemma:appendixd}}\lab{sec:appendixd}

Note that 
\bee
u(\eta(\sigma),\o_1)' &=& \gg(\dd u(\eta(\sigma),\o_1), \eta'(\sigma))\\
&=& 0
\eee
by the definition of $\dd u$, $\eta'$ and the identity \eqref{identities1}. This implies
$$u(\eta(\sigma),\o_1)=u(\gamma_{\o_1}(s-t),\o_1)=u(t,x,\o_1),$$
which together with the fact that $\o_1\in D$ implies from the definition \eqref{eo1} of $D$
\be\lab{oo1}
u(\eta(\sigma),\o_1)=u(s,y,\o_1)\textrm{ for all }\sigma\in\R.
\ee
Also, we have
\bee
\po u(\eta(\sigma),\o_1)' &=& \gg(\dd \po u(\eta(\sigma),\o_1), \eta'(\sigma))\\
\nn&=& \gg(-b^{-2}\po b N(\eta(\sigma),\o_1)+b^{-1}\po N(\eta(\sigma),\o_1), \eta'(\sigma))\\
\nn&=& \po u(s,y,\o_1)-\po u(\gamma_{\o_1}(s-t),\o_1),
\eee
where we used in the last inequality the definition \eqref{ooter} of $\eta_{\o}'$ and the definition \eqref{oobis} of $a_1$. This implies
\be\lab{oo2}
\po u(\eta(\sigma),\o_1)=\po u(\gamma_{\o_1}(s-t),\o_1)+\sigma(\po u(s,y,\o_1)-\po u(\gamma_{\o_1}(s-t),\o_1)).
\ee
In view of \eqref{oo1} and \eqref{oo2}, we have
$$u(\eta(1),\o_1)=u(s,y,\o_1),\, \po u(\eta(1),\o_1)=\po u(s,y,\o_1).$$
Since $u(s,.,\o_1), \po u(s,.,\o_1)$ forms a global coordinate system on $\Sigma_s$ in view of the assumption in section \ref{sec:assumptioncoord}, we deduce
$$(s,y)=\eta(1)$$
which is the desired estimate. This concludes the proof of Lemma \ref{lemma:appendixd}.

\end{document}